\documentclass[12pt]{amsart}
\usepackage{amscd}
\usepackage{lscape}
\usepackage{euscript}
\usepackage{amsfonts}
\usepackage{amsmath, amscd, amssymb}
\usepackage{latexsym}
\def \R{\mathbb{R}}
\def \GW{G\mathcal{W}}
\def \C{\mathbb{C}}
\def \Z{\mathbb{Z}}

\def \D{\mathrm{Diff}}
\def \BSO{\mathop{\mathrm{BSO}}}
\def \BO{\mathop{\mathrm{BO}}}
\def \SO{\mathop{\mathrm{SO}}}
\def \BDiff{\mathop{\mathrm{BDiff}}}
\def \coker{\mathop{\mathrm{coker}}}
\def \coim{\mathop{\mathrm{coim}}}
\def \im{\mathop{\mathrm{im}}}
\def \O{\mathop{\mathrm{O}}}

\def \dim{\mathrm{dim}\,}
\def \id{\mathrm{id}}

\def \A{\mathbf{A}}
\def \B{\mathbf{B}}
\def \X{\mathbf{X}}
\def \Y{\mathbf{Y}}
\def \MSO{\mathbf{MSO}}
\def \K{\mathcal{K}}

\input xy
\xyoption{all}

\newtheorem{theorem}{Theorem}[section]

\newtheorem{lemma}[theorem]{Lemma}
\newtheorem{proposition}[theorem]{Proposition}
\newtheorem{corollary}[theorem]{Corollary}

\theoremstyle{remark}
\newtheorem{remark}[theorem]{Remark}

\theoremstyle{definition}
\newtheorem{definition}[theorem]{Definition}

\newtheorem{example}[theorem]{Example}

\title{Fold maps, framed immersions and smooth structures}
\author{R. Sadykov}
\thanks{The author has been supported by the
FY2005 Postdoctoral Fellowship for Foreign Researchers, Japan Society for the Promotion of Science; and a 
Postdoctoral Fellowship of Max Planck Institute, Germany}
\address{Department of Mathematics, University of Toronto, Canada}
\email{rstsdk@gmail.com}

\begin{document}
\begin{abstract}
For each integer $q\ge 0$, there is a cohomology theory $\mathbf{A}_1$ such that the zero cohomology group $\mathbf{A}_1^0(N)$ of a manifold $N$ of dimension $n$ is a certain group of cobordism classes of proper fold maps of manifolds of dimension $n+q$ into $N$.  We prove a splitting theorem  for the spectrum representing the cohomology theory of fold maps. For even $q$, the splitting theorem implies that the cobordism group of fold maps to a manifold $N$ is a sum of $q/2$ cobordism groups of framed immersions to $N$ and a group related to diffeomorphism groups of manifolds of dimension $q+1$. Similarly, in the case of odd $q$, the cobordism group of fold maps splits off $(q-1)/2$ cobordism groups of framed immersions.

The proof of the splitting theorem gives a partial splitting of the homotopy cofiber sequence of Thom spectra in the Madsen-Weiss approach to diffeomorphism groups of manifolds. 
\end{abstract}
\maketitle

\section{Introduction}

A point $x\in M$ is said to be a {\it fold singular point} of a smooth map $f:M\to N$ of manifolds of dimensions $n+q$ and $n$ respectively with $q\ge 0$ if there are coordinate neighborhoods of $x$ in $M$ and $f(x)$ in $N$ with respect to which the map $f$ is the product of a Morse function and an identity map, i.e., 
\[
    f: (\R^{n+q}, 0)\longrightarrow (\R^n, 0),
\]
\[
     f\colon (x_1, ..., x_{n+q}) \mapsto (x_1, ..., x_{n-1}, \pm x^2_{n} \pm \cdots \pm x^2_{n+q}).
\]
A map with only fold singular points is called a {\it fold map}.   
Fold maps and their modifications are known to give extensive information about manifolds. They are related to stable spans~\cite{OSS}, \cite{SSS}, smooth structures on $4$-manifolds \cite{SS}, smooth structures on spheres \cite{Sae1}, \cite{Sa5}, foliations~\cite{ElMi}, pseudoisotopies~\cite{Ig} and such theorems in algebraic topology as the Barratt-Priddy and Quillen theorem~\cite{Fu}, standard Mumford conjecture~\cite{MW}, May-Segal theorem~\cite{Va} and Hopf invariant one problem~\cite{KS}.  

For each integer $q\ge 0$, there is a cohomology theory $\mathbf{A}_1$ such that the zero cohomology group $\mathbf{A}_1^0(N)$ of a manifold $N$ of dimension $n$ is a certain group of cobordism classes of proper fold maps of manifolds of dimension $n+q$ into $N$.
In the current paper for each $q\ge 0$ we define cobordism theories  $\B_i(\cdot)$, with $0\le i\le \frac{q+1}{2}$, and two cohomology theories $\X(\cdot)$ and $\Y(\cdot)$. The cobordism theories $\B_i$ are related to stable homotopy groups of classifying spaces of certain Lie groups (see Proposition~\ref{p:4.3}); while the cohomology theory $\X$ is related to diffeomorphism groups of manifolds of dimension $q+1$. The representing spectrum $\mathbf{X}$ is also known as $\mathbf{hV}$, and, specifically for $q=1$, as $\C P^{\infty}_{-1}$ (e.g., see \cite{MW}). 

The first main result of the paper is the splitting theorem for cohomology theory of fold maps. 

\begin{theorem}[Splitting Theorem]\label{main}
The spectrum $\mathbf{A}_1$ of the cobordism theory of fold maps of codimension $-q$ is equivalent to 
\begin{equation}\label{eq:10}
  \mathbf{A}_1 \simeq   \B_1\vee \cdots \vee \B_{q/2} \vee \X 
\end{equation}
if $q$ is even, and to
\begin{equation}\label{eq:11}
 \mathbf{A}_1\simeq    \B_1\vee \cdots \vee \B_{(q-1)/2} \vee \Y 
\end{equation}
if $q$ is odd. 
\end{theorem}

The existence of the splittings (\ref{eq:10}) and (\ref{eq:11}) significantly clarifies the structure of the spectrum of cobordism groups of fold maps
and remarkably simplifies the study of fold maps (e.g., compare with \cite{An3}). In particular we reduce the study of fold maps to simpler and more familiar objects: the spectrum $\mathbf{X}$ has been studied \cite{GMTW} in relation to diffeomorphism groups of manifolds, while for a study of cobordism theories $\B_i$ we refer to the book of Koschorke \cite{Ko}.   

Recently fold maps were of much interest to topologists in relation to diffeomorphism groups of surfaces (e.g., see the papers \cite{MW} of Madsen and Weiss and \cite{EGM} of Eliashberg, Galatius and Mishachev). Let $F_{g,1}$ be an oriented compact surface of genus $g$ with one boundary component. Then there is an inclusion $F_{g,1}\subset F_{g+1,1}$ and every orientation preserving diffeomorphism of $F_{g,1}$ pointwise trivial on the boundary $\partial F_{g,1}$ extends to a diffeomorphism of $F_{g+1,1}$ as the identity on $F_{g+1,1}\setminus F_{g,1}$. In particular, there is a sequence of inclusions 
\[
    \BDiff F_{0,1}\longrightarrow \BDiff F_{1,1}\longrightarrow \BDiff F_{2,1}\longrightarrow \cdots 
\] 
of classifying spaces whose union is denoted $\BDiff F_{\infty,1}$. The generalized Mumford Conjecture identifies the homotopy type of the Quillen ``+"-cons\-truction $\Z\times \BDiff F_{\infty,1}^+$, it was solved by considering  a map of spectra 
\begin{equation}\label{equ:1}
\mathbf{hW}\longrightarrow \B_0 \vee \B_1\vee \cdots \vee \B_{q+1},
\end{equation}
where $\mathbf{hW}$ is a version of our $\mathbf{A_1}$ but for \emph{cooriented} fold maps.
 
More precisely, Madsen and Weiss solved \cite{MW} the Mumford Conjecture by proving that in the case $q=2$, the homotopy fiber of the map (\ref{equ:1}) is homotopy equivalent to $\Z\times \BDiff F_{\infty,1}^+$.

The proof of the splitting for the cohomology theory of fold maps shows that in the case where $q$ is even, the map of spectra (\ref{equ:1}) considered by Madsen and Weiss partially splits. Namely, the following theorem takes place.

\begin{theorem} For $q$ even, there is a splitting 
\[
 \mathbf{hW}\simeq \B_0 \vee \B_1 \vee \cdots \vee \B_{q} \vee \mathbf{X}.
\]
\end{theorem}

A referee pointed out that a version of our splitting theorem was also proved by Ando in \cite[Propositions 10.3 and 10.4]{An4}. Our works are independent and the result first appeared in our paper \cite{Sad0}.

\subsection*{Applications}
The results of this paper has been used to  compute rational Morin cobordism groups~\cite{Sa1}. 


\subsection*{Acknowledgment} The author is thankful to O. Saeki and B. Kalmar for comments and questions; Proposition~\ref{l:15} and Proposition~\ref{p:5.2} were proved in attempt to answer some of their questions.

\section{Cobordism groups of maps with prescribed singularities}\label{s:2}

In this section we recall the definition and properties of spectra for cobordism groups of maps with singularities of prescribed types. Except for Examples~\ref{ex:1}, \ref{ex:2}, in this paper we will tacitly assume that the dimension of the target manifold $N$ of cobordism groups under consideration is greater than $1$.

\subsection{Singularity types}
For a smooth map $f:M\to N$ of manifolds, a point $x\in M$ is said to be {\it singular} if the rank of the differential $df$ at $x$ is less than the minimum of $\dim M$ and $\dim N$. A non-singular point of $f$ in $M$ is called {\it regular}. The \emph{dimension} of $f$ is defined to be the number $\dim M-\dim N$. Two continuous maps $f,g\colon X\to Y$ of topological spaces define the same \emph{map germ} at $x\in X$ if there is a neighborhood $U\subset X$ of $x$ such that $f|U=g|U$. 

For $i=1,2$, let $X_i$ and $Y_i$ be smooth manifolds. We say that a smooth map germ $f_1: X_1\to Y_1$ at $x_1$ has the same \emph{singularity type} as a smooth map germ $f_2: X_2\to Y_2$ at $x_2$ if there are neighborhoods $U_i\subset X_i$ of $x_i$ and $V_i\subset Y_i$ of $f_i(x_i)$ and diffeomorphism germs $\alpha, \beta$ that fit a commutative diagram of map germs
\[
\begin{CD}
(U_1\times \R^t, x_1\times\{0\}) @>\alpha>> (U_2\times \R^s, x_2\times \{0\}) \\
@Vf_1\times\id_{\R^t} VV @V f_2\times \id_{\R^s}VV \\
(V_1\times \R^t, f_1(x_1)\times\{0\}) @>\beta>> (V_2\times \R^s, f_2(x_2)\times \{0\}),
\end{CD}
\]
where $\id_{\R^t}$ and $\id_{\R^s}$ are the identity maps of $\R^t$ and $\R^s$ respectively. We note that the singularity type of a map germ determines the dimension of the map germ, but it does not determine the dimensions of the source and target manifolds. 

Given a prescribed set $\tau$ of singularity types of map germs of dimension $q$, a {\it $\tau$-map} is defined to be a map of dimension $q$ each map germ of which is of singularity type $\tau$. We say that a set $\tau$ of singularity types is {\it open} if a $C^{\infty}$-slight perturbation of any $\tau$-map is also a $\tau$-map.  In particular, every non-empty open set of singularity types contains the singularity type of a regular map germ. An {\it orientation} on a $\tau$-map $f: M\to N$ is defined to be a choice of orientation on the normal bundle $f^*TN\ominus TM$ of $f$, where $TM$ and $TN$ stand for the tangent bundles of $M$ and $N$ respectively, and $f^*TN\ominus TM$ is a stable vector bundle of dimension $-q$ (see Definition~\ref{d:v}). 

In this paper we will consider only oriented $\tau$-maps. 

Two proper $\tau$-maps $f_i: M_i\to N$, with $i=0,1$, are said to be {\it cobordant} if there is a proper $\tau$-map $f: M\to N\times [0,1]$ of a manifold $M$ with boundary $\partial M=M_0\sqcup (-M_1)$ such that $f|M_i=f_i$ for $i=0,1$, and the restriction of $f$ to collar neighborhoods of $M_0$ and $M_1$ in $M$ can be identified with the disjoint union of maps 
\[
    f_0\times \id\colon M_0\times [0, \delta)\longrightarrow N\times [0, \delta) 
\]    
and 
\[
    f_1\times \id\colon M_1\times (1-\delta, 1]\longrightarrow N\times (1-\delta, 1], 
\]    
by means of respective identifications of $M_0\times [0, \delta)$ and $M_1\times (1-\delta, 1]$ with collar neighborhoods of $M_0$ and $M_1$ in $M$. Here
 $\id$ is the identity map of an appropriate space, and $\delta>0$ is a sufficiently small real number.
 Cobordism classes of proper $\tau$-maps into a manifold $N$ form a monoid with respect to disjoint union which by the Grothendieck construction 
can be extended to a group called the {\it cobordism group of $\tau$-maps} into $N$. 

\begin{remark} The construction of a cohomology theory associated with $\tau$-maps in subsection~\ref{subs:2.2} shows that in general one should consider not the Grothendieck group completion but a certain localization of the monoid of cobordism classes of $\tau$-maps. For fold maps the two approaches yield the same result, but, for example, for coverings the results are different (for details, see \cite{Sady}).
\end{remark}

Cobordism groups of $\tau$-maps into smooth manifolds form a group valued functor. It is defined on the category of smooth manifolds without boundary and dimension $0$-embeddings.  Its value on a manifold $N$ is the cobordism group of $\tau$-maps to $N$. 

This functor is contravariant. Let $f\colon M\to N$ be a proper $\tau$-map representing an element of the cobordism group of $\tau$-maps into $N$. Then a dimension $0$-embedding $N'\subset N$ gives rise to a proper $\tau$-map $f|f^{-1}(N')$ representing an element of the cobordism group of $\tau$-maps into $N'$. In turn, the correspondence $f\mapsto f|f^{-1}(N')$ determines a homomorphism of cobordism groups of $\tau$-maps, which is the value of the contravariant functor on the inclusion $N'\subset N$. We will see later (by Theorem~\ref{th:0} below) that under certain conditions this functor can be extended to take values on continues maps, which is somewhat surprising.

\begin{example}\label{ex:1} The cobordism group of non-singular maps of dimension $q$ into a one-point-space $P$ is a free group whose generators are in bijective correspondence with diffeomorphism types of manifolds of dimension $q$. 
\end{example}

\begin{example}\label{ex:2}
The cobordism group of fold maps of dimension $q$ into a one-point-space $P$ is isomorphic to the cobordism group $\Omega_q$ of Thom of oriented manifolds of dimension $q$ under the isomorphism associating to the cobordism class of a fold map $M\to P$ the cobordism class of $M$. 
\end{example}

\subsection{Cohomology theory of $\tau$-maps}\label{subs:2.2}
For a fairly general choice of $\tau$ the functor of the cobordism group of $\tau$-maps actually admits an extension to a generalized cohomology theory (see Theorem~\ref{th:0} below) represented by a spectrum $\Sigma_{\tau}$. 

The spectrum $\Sigma_{\tau}$ is defined for an arbitrary not necessarily open set $\tau$ of singularity types. Let $\xi_t: E_t\to \BSO_t$ be the universal vector bundle of dimension $t$ over the space $\BSO_t$ of oriented vector subspaces of $\R^{\infty}$ of dimension $t$. Let $S_t=S_t(\tau)$ denote the space of map germs at $0$ of maps $f\colon \R^{t+q}\to E_t$ such that 
\begin{itemize}
\item the image of $f$ belongs to a single fiber $E_t|b$ of $\xi_t$ over some point $b\in \BSO_t$,
\item $f(0)$ is the zero in the vector space $E_t|b$, and
\item the map germ $f\colon (\R^{t+q}, 0)\to (E_t|b, 0)$ is a smooth $\tau$-map germ.  
\end{itemize}
The space $S_t$ is endowed with an obvious topology so that the map $\pi_t: S_t\to \BSO_t$ that takes $f$ onto $b$ has a structure of a fiber bundle. Equivalently, the space $S_t$ is the total space of the fiber bundle given by the Borel construction
\[
      {\mathop\mathrm{ESO}}_t\times_{SO_t} S^0_t\longrightarrow {\BSO}_t,
\]
where ${\mathop\mathrm{ESO}}_t\to \BSO_t$ is the universal principle $\SO_t$-bundle, and $S^0_t\subset C^{\infty}(\R^{t+q},\R^t)$ is the subspace of smooth maps $f$ such that $f(0)=0$ and the map germ of $f$ at $0$ is a $\tau$-map germ. Here an element $g\in \SO_t$ acts on the space $C^{\infty}(\R^{t+q},\R^t)$ of smooth maps, and therefore on $S^0_t$, by sending a map $f$ onto the composition $g\circ f$. 
The $(t+q)$-th term of the spectrum $\Sigma_{\tau}$ is
defined to be the Thom space $T\pi_t^*\xi_t$ of the bundle $\pi_t^*\xi_t$
over $S_t$.

The cohomology theory of ${\tau}$-maps can also be described in terms of stable vector bundles. 

\begin{definition}\label{d:v}
A \emph{stable vector bundle} of dimension $q$, with $q\in \Z$, over a CW complex $X$ is defined to be a map $X\to \{q\}\times \BO$. Every stable vector bundle $f\colon B\to \{-q\}\times\BO$ of dimension $-q$ determines a cohomology theory $h^*$; the $(n+q)$-th term of the spectrum of $h^*$ is the Thom space of the vector bundle $(f|B_n)^*\xi_n$ where $B_n=f^{-1}(\BO_n)$. Similarly, an oriented stable vector bundle over a CW complex $X$ of dimension $q$ is defined to be a map $f\colon B\to \{-q\}\times \BSO$, and every oriented vector bundle  determines a cohomology theory. 

 Two stable vector bundles over a CW complex $X$ are \emph{isomorphic} if they are homotopic as maps $X\to\{q\}\times \BO$. For a finite CW complex $X$ isomorphism classes of stable vector bundles over $X$ are in bijective correspondence with equivalence classes of pairs $(\xi, \eta)$ of vector bundles over $X$ with $q=\dim \xi-\dim \eta$, where two pairs $(\xi,\eta)$ and $(\xi',\eta')$ are equivalent if $\xi\oplus\eta'\oplus \varepsilon^r$ is isomorphic to $\xi'\oplus \eta\oplus \varepsilon^r$ for some trivial vector bundle $\varepsilon^r$ over $X$. Isomorphic stable vector bundles give rise to equivalent cohomology theories. To simplify notation we will use the same notation for stable vector bundles and their equivalence classes. The stable vector bundle of a pair $(\xi,\eta)$ is denoted by $\xi\ominus\eta$.  
\end{definition}

To describe the stable vector bundle corresponding to the cohomology theory of $\tau$-maps, consider commutative diagrams
\[
   \begin{CD}
   S_t@>j_t>> S_{t+1}\\
   @V\pi_tVV @VV\pi_{t+1}V \\
   \BSO_t @>i_t>>\BSO_{t+1}
   \end{CD}
\]
where $t$ ranges over non-negative integers, the map $i_t$ is the canonical inclusion, and the top horizontal map $j_t$ sends a map germ $f\colon \R^{t+q}\to E|b$ onto the map germ of the composition
\[
   \R^{t+q+1}\equiv \R^{t+q}\times \R\stackrel{f\times \id_{\R}}\longrightarrow E_t|b\times \R\stackrel{\approx}\longrightarrow E_{t+1}|i_t(b),
\] 
(for details see \cite{Sa}). The last isomorphism in the composition is given by the splitting $\xi_{t+1}|\BSO_t\approx\xi_t\oplus\varepsilon$. Here and below $\varepsilon$ stands for the trivial line bundle over an appropriate space. We will denote the direct limit $\lim S_t$ of the inclusions $j_t$ by $S(\tau)$ and the map $\lim \pi_t$ by $\pi(\tau)$. Thus, $\pi(\tau)$ is a map $S(\tau)\to \{-q\}\times \BSO$. We will write $\pi$ for $\pi(\tau)$ if $\tau$ is understood.    

For future reference we reformulate the definition of the cohomology theory $\Sigma_{\tau}$. 

\begin{proposition}\label{r:3.1} The cohomology theory $\Sigma_{\tau}$ corresponds to the oriented stable vector bundle $\pi\colon S(\tau)\to \{-q\}\times \BSO$ of dimension $-q$. 
\end{proposition}

\section{Submersions}\label{ex:3}
In this section we show that the spectrum $\mathbf{A_0}$ of regular maps, i.e., the spectrum $\Sigma_{\tau}$ where $\tau$ contains only the singularity type of a regular map germ, is equivalent to the spectrum defined in the paper \cite{GMTW}, i.e., the spectrum corresponding to the stable vector bundle $\BSO_q\to \{-q\}\times \BSO$ classifying $-\xi_q$.  

In this case the space $S_t(\tau)$ will be denoted by $A_{0,t}$. 

\begin{lemma}\label{lemma:3.1} The space $A_0=\lim A_{0,t}$ is homotopy equivalent to $\BSO_t$.
\end{lemma}
\begin{proof}
The space $A_{0,t}$ is homotopy equivalent to the total space of the fiber bundle over $\BSO_t$ with fiber over $b\in \BSO_t$ given by the space of surjective homomorphisms $\R^{t+q}\to E_t|b$ of vector spaces. There is a fibration map $p$ from $A_{0,t}$ to the Grassmannian manifold $G_q(\R^{t+q})$ of subspaces of $\R^{t+q}$ of dimension $q$. It takes a point $f$ in $S_t$ onto the point in the Grassmannian manifold corresponding to the kernel of $df$ at $0\in \R^{t+q}$. Let $L\subset \R^{t+q}$ be a subspace of dimension $q$, it corresponds to a point in $G_q(\R^{t+q})$. Let $L^{\perp}\subset \R^{t+q}$ denote the subspace orthogonal to $L$. Then the fiber of $p$ over $L$ consists of all isomorphisms from $L^{\perp}$ to the fibers of $\pi_t$; hence, the fibers of $\pi_t$ are contractible.    
Clearly $A_0=\lim A_{0,t}$ is the classifying space $\BSO_q$. 
\end{proof}

Next we need to determine the homotopy class of the stable vector bundle map $\pi|A_0\colon A_0\to \{-q\}\times \BSO$.  Our argument (which we will repeatedly use later) is in terms of vector bundles over spaces $S_t(\tau)$, where $\tau$ consists of a single singularity type. Let us recall that for a smooth map $f: M\to N$ of manifolds, there is an exact sequence 
\begin{equation}\label{e:1}
    0\longrightarrow \ker(df)\longrightarrow TM\stackrel{df}\longrightarrow f^*TN\longrightarrow \coker(df)\longrightarrow 0
\end{equation}
of vector bundles where all vector bundles are restricted to a subset $\Sigma$ of $M$ of singular points of $f$ of a single prescribed type. This sequence splits and we have isomorphisms of vector bundles
\begin{equation}
     TM\longrightarrow \ker(df)\oplus \coim(df) \qquad \mathrm{over } \quad \Sigma,
\end{equation}
\begin{equation}\label{e:3}
    \coker(df)\oplus \im(df)\longrightarrow f^*TN \qquad \mathrm{over } \quad \Sigma,
\end{equation}
and a fiberwise isomorphism of vector bundles $\coim(df)\to \im(df)$. There are similarly defined vector bundles over $S_t(\tau)$ if $\tau$ consists of a single singularity. For example, the cokernel bundle is defined to be the subbundle of $\pi_t^*\xi_t$ with fiber over $f\in S_t$ given by the pullback of the cokernel of $df$, while the kernel bundle is defined to be the subbundle of $\pi_t^*\varepsilon^{t+q}$ with fiber over $f$ given by the pullback of the kernel of $df$. Here $\varepsilon^{t+q}$ is the trivial vector $(t+q)$-bundle over $\BSO_t$. The kernel and cokernel bundles over $S_t$ give rise to the canonical kernel and cokernel bundles over $S(\tau)$.

For example, the canonical kernel bundle $K_0$ over $A_0$ is isomorphic to the universal vector bundle of dimension $q$ over $\BSO_q$ under an obvious isomorphism covering the colimit $A_0\to \BSO_q$ of fibrations $p$ constructed in the proof of Lemma~\ref{lemma:3.1}.  

We recall that the space $\BSO$ is the colimit of Grassmann manifolds $G_{i,j}$ of $i$-planes in $\R^{i+j}$, and there is an involution $I$ on $\BSO$ described in~\cite{St} given by the colimit of maps $G_{i,j}\to G_{j,i}$ taking an $i$-plane onto its orthogonal $j$-plane.

\begin{proposition}\label{l:2} The map $\pi|A_0$ is defined by the composition 
\[ 
      {\BSO}_q \longrightarrow \{q\}\times \BSO \longrightarrow \{-q\}\times \BSO 
\]
of the inclusion and the involution $I$. In other words the map $\pi|A_0$ is a stable vector bundle over $\BSO_q$ isomorphic to the negative of the canonical kernel bundle over $A_0$.
\end{proposition}
\begin{proof} The $(t+q)$-th space of the spectrum $\mathbf{A_0}$ is the Thom space of the vector bundle 
$(\pi_t|A_{0,t})^*\xi_t$. Hence the map $\pi_t|A_{0,t}$ classifies the stable vector bundle $\pi_t^*(\xi_t\ominus \varepsilon^{t+q})$ of dimension $-q$ over $A_{0,t}$ (see Proposition~\ref{r:3.1}). On the other hand, there is an exact sequence of vector bundles
\[
     0\longrightarrow K_{0,t}\longrightarrow (\pi_t|A_{0,t})^*\varepsilon^{t+q} \longrightarrow (\pi_t|A_{0,t})^*\xi_t\longrightarrow 0, 
\]
where $K_{0,t}$ is the canonical kernel bundle over $A_{0,t}$ and the third map over a point $f\in A_{0,t}$ is induced by the differential $df$. In particular, the stable vector bundle $-K_{0,t}$ of dimension $-q$ is isomorphic to the stable vector bundle $(\pi_t|A_{0,t})^*(\xi_t\ominus \varepsilon^{t+q})$. Since the inclusion of $\BSO_q$ into $\{q\}\times \BSO$ classifies the canonical kernel bundle $K_0$ over $A_0$, this implies the statement of the lemma. 
\end{proof}

\begin{corollary} The spectrum $\mathbf{A_0}$ coincides with the spectrum of the stable vector bundle $-\xi_q: \BSO_q\to \{-q\}\times \BSO$ of dimension $-q$. 
\end{corollary}

\section{Properties of the spectrum $\Sigma_{\tau}$} In this section we list properties of $\Sigma_{\tau}$. In particular, we recall the homotopy type of $S(\tau)$ in the case where $\tau$ consists of a single singularity type, and the isomorphism class of the canonical kernel bundle over $S(\tau)$. 

For the next theorem we need the notion of an open set $\tau$ of singularity types being \emph{$\mathcal{K}$-invariant} (e.g., see \cite{Sa}). This condition is satisfied when $\tau$ consists of the singularity types of fold and regular map germs. We will not recall the definition here, but quote the following theorem.

\begin{theorem}\label{th:0} Let $\tau$ be an open set of $\K$-invariant singularities imposed on maps of dimension $q\ge 0$. Suppose that $\tau$ contains fold singularities. Then for every closed manifold $N$ of dimension at least $2$, the cobordism group of $\tau$-maps into $N$ is isomorphic to the limit 
\[
    \Sigma_{\tau}(N)\colon = [\mathbf{S}\wedge N_+, \Sigma_{\tau}] =\lim_{t\to \infty} [S^{t}\wedge N_+, [\Sigma_{\tau}]_{t}]
\]  
of groups of pointed homotopy classes of maps, where $\mathbf{S}$ is the sphere spectrum, and $[\Sigma_{\tau}]_t$ is the $t$-th space $T\pi_{t-q}^*\xi_{t-q}$ of the spectrum $\Sigma_{\tau}$. 
\end{theorem}

\begin{remark} In the stated form Theorem~\ref{th:0} is proved by the author in \cite{Sa}. A dual version (for bordisms of $\tau$-maps) of Theorem~\ref{th:0} with an essentially different spectrum is established by Ando in \cite{An4}. Theorem~\ref{th:0} is true for maps of negative dimension as well. In fact in this case it is true even without the assumption that $\tau$ contains fold singularities, and in this case versions of Theorem~\ref{th:0} are due to Ando~\cite{An4}, the author~\cite{Sa}, and Sz\H{u}cs~\cite{Sz1}. The assumption on fold singularities is not necessary also in the case $q=0$; and the assumptions $\dim N\ge 2$ and that $N$ is closed can be omitted \cite{Sady}. The spectrum $\Sigma_{\tau}$ is dual to a spectrum introduced by Eliashberg in \cite{El} (which is an equivariant spectrum for bordisms of $\tau$-maps) and is a generalization of Wells spectrum for immersions~\cite{We} and the spectrum $\A_0$ of Galatius, Madsen, Tillmann and Weiss~\cite{GMTW} (see subsection~\ref{ex:3}). It is also closely related to two general constructions of classifying spaces by Rimanyi-Sz\H{u}cs~\cite{RS} and Kazarian~\cite{Kaz}, \cite{SY}. 
\end{remark}

\begin{example} For the set $\tau$ of fold singularity types and the singularity type of a regular map germ, we will denote the spectrum $\Sigma_{\tau}$ by $\A_1$. If $N$ is a smooth closed manifold, then the group $\A_1^0(N)=\Sigma_{\tau}(N)$ is isomorphic to the cobordism group of oriented fold maps into $N$.  
\end{example}


\begin{remark} The generalized cobordism theory of $\Sigma_{\tau}$ associates cohomology groups not only to closed manifolds, but also to general topological spaces, and, in particular, to open smooth manifolds. We caution the reader that some authors define cobordism groups of fold maps into open manifolds $N$, and in particular into $\R^n$, in a way different from ours; namely, not as $\Sigma_{\tau}(N):=\lim_t[S^t\wedge N_+, [\Sigma_{\tau}]_t]$, where $N_+$ stands for $N\sqcup \{*\}$, but as $\Sigma_{\tau}(\dot{N}, *)$, where $\dot{N}$ denotes a one point compactification of $N$. 
\end{remark}

Let $\tau$ be the singularity type of a finitely determined map germ $h\colon (\R^{t+q}, 0)\to (\R^t, 0)$ such that for every map germ $f$ of singularity type $\tau$ there are local coordinates with respect to which $f=h\times \id$, where $\id$ is the identity map germ of $(\R^k, 0)$ for some non-negative integer $k=k(f)$.  For $i \ge 0$ let $\D(\R^i, 0)$ denote the group of diffeomorphism
germs of $(\R^i, 0)$. The group $\mathop\mathrm{Stab}(h)$ is defined to be the subgroup of the group $\D(\R^{t+q}, 0) \times \D(\R^t, 0)$ of elements $(\alpha, \beta)$ such that $\beta\circ h\circ\alpha^{-1}=h$. The \emph{relative symmetry group} of $\tau=\tau(h)$ is a maximal compact subgroup of the group
\begin{equation}\label{eq:group}
\{\,(\alpha, \beta) \in \mathop\mathrm{Stab}(h)\, |\, \det(d\alpha|0) \cdot \det(d\beta|0) > 0 \,\},
\end{equation}
where $d\alpha$ and $d\beta$ are differentials of $\alpha$ and $\beta$ respectively. By the J\"{a}nich-Wall theorem~\cite{Ja}, \cite{Wa}, a maximal compact subgroup of (\ref{eq:group}) exists and any two maximal compact subgroups are conjugate. The \emph{right} and \emph{left} \emph{representations} of the relative symmetry group of $\tau$ are its representations on $(\R^{t+q}, 0)$ and $(\R^t, 0)$ given by the restrictions of the projections of $\D(\R^{t+q}, 0)\times \D(\R^t, 0)$ onto the first and second factors respectively. 

\begin{lemma}[Kazarian-Sz\H ucs, \cite{Kaz}, \cite{Sz1}]\label{l:0} The space $S(\tau)$ is homotopy equivalent to the classifying space $BG$ of the relative symmetry group $G$ of the singularity type $\tau$. 
\end{lemma}

\begin{example} A map $f:M\to N$ is said to be an \emph{immersion} if it is a map of non-positive dimension $q=-d$  and $\mathop{\mathrm{rank}}df=\dim M$. In this case the space $S_t(\tau)$ is homotopy equivalent to the total space of the fiber bundle over $\BSO_t$ with fiber over $b\in \BSO_t$ given by the space of $(t-d)$-frames in the vector space $E_t|b$. Clearly $S_t=\BSO_d$ for $t\ge d$ and $S(\tau)=BG$ for the symmetry group $G=\SO_d$.  
\end{example}

\begin{remark}\label{r:2.5}
The relative symmetry groups of fold singularities are computed for example in the paper ~\cite{Sa1} and can be described as follows. For non-negative integers $a,b$, let $\SO(a,b)$ denote the group of orientation preserving elements of the product of orthogonal groups $\O_a\times \O_b$. 
Then the relative symmetry group $G$ of the fold singularity of index $a$ is a subgroup of $\O_{a+b}$ (or $\O_{2a}$) generated by 
\begin{itemize}
\item $\SO(a,b)$, where $a+b=q+1$, if $a\ne b$;   
\item $<\SO(a,b), r_a\circ h_a>$, if $a=(q+1)/2$ and $a$ is even; and
\item $<\SO(a,b), h_a>$, if $a=(q+1)/2$ and $a$ is odd.
\end{itemize}
Here, for $a>0$, the element $h_a$ stands for the transformation in $\O_{2a}$ that exchanges the two factors of $\R^a\times \R^a$; and $r_a$ is a fixed element in $\O_{2a}$ that reflects the first factor of $\R^a\times \R^a$ along a hyperplane. 
\end{remark}

\begin{lemma}\label{l:2.4}
In the case where $\tau$ is any fold singularity or the singularity type of a regular map germ, the canonical kernel bundle over $S(\tau)$ is the universal vector $G$-bundle over $BG$. 
\end{lemma}
\begin{proof} 
The statement follows from the observation that the action of the group $G$ on the limit of kernel bundles is the standard one. The latter can be verified for example by inspecting the list of relative symmetry groups given in Remark~\ref{r:2.5}.
\end{proof}

Let $\tau$ be the set that contains only the singularity type of the regular map germ of dimension $q$; and $\eta$ the set that consists of $\tau$ and the singularity type of the Morse map germ of index $0$. It is known that the canonical kernel bundle $K_1$ over $S(\eta)\setminus S(\tau)$ is isomorphic to the normal bundle of $S(\eta)\setminus S(\tau)$ in $S(\eta)$; e.g. see \cite{Bo}. Hence, by Lemmas~\ref{l:0} and \ref{l:2.4}, the normal bundle of $S(\eta)\setminus S(\tau)$ in $S(\eta)$ is a model for the universal vector bundle $E_{q+1}\to \BSO_{q+1}$. Explicitly, the isomorphism can be constructed as follows. A point in $S(\eta)\setminus S(\tau)$ is represented by a fold map germ $f\colon (\R^{t+q}, 0)\to (\xi_t|b, 0)$, while the fiber of the normal bundle of $S(\eta)\setminus S(\tau)$ in $S(\eta)$ over $f$ is given by the space of map germs
\[
   f+v^*\colon (\R^{t+q}, 0)\longrightarrow (\xi_t|b, 0)
\]
where $v^*$ is the composition of the orthogonal projection $\R^{t+q}\to L$ onto $L=\mathop\mathrm{ker}df|0$ and a linear function on $L$ with values in $\R\simeq \mathop\mathrm{coker}df|0$.  By using the standard Riemannian metric we may identify $v^*$ with the dual vector $v\in L$. Then in terms of representatives the correspondence $(f, v^*)\mapsto (L, v)$ defines a map to $E_{q+1}$ from a neighborhood of $S(\eta)\setminus S(\tau)$ in $S(\eta)$ of pairs $(f, v^*)$ with $||v^*||\ll \min ||f(u)||$ where $u$ ranges over vectors in $\R^{t+q}$ orthogonal to $L$.

\begin{lemma}\label{l:3.6} Let $\tau$ be the set that contains only the singularity type of the regular map germ; and $\eta$ the set that consists of $\tau$ and the singularity type of the Morse map germ of index $0$. Then the inclusion $\tau\subset \eta$ induces a homotopy equivalence $S(\eta)\setminus S(\tau)\subset S(\eta)$. 
\end{lemma}
\begin{proof} Let $\mathbf{0}\subset E_{q+1}$ denote the image of the zero section in the total space $E_{q+1}$ of the universal vector bundle over $\BSO_{q+1}$. Then there is an equivalence
\[
      S(\eta)= S(\tau)\sqcup E_{q+1}/\sim
\]
where $\sim$ identifies the two copies of $E_{q+1}\setminus \mathbf{0}$ in $S(\tau)$ and $E_{q+1}$. By Lemma~\ref{l:0} the set $S(\tau)$ is homotopy equivalent to the classifying space $\BSO_q$. On the other hand it is well-known that $E_{q+1}\setminus \mathbf{0}$ is a model for the classifying space $\BSO_q$. Thus it remains to show that the inclusion 
\begin{equation}\label{eq:3.5}
 E_{q+1}\setminus \mathbf{0}\longrightarrow S(\tau)
\end{equation}
is a homotopy equivalence. To this end, let us identify the universal vector bundles over the two copies of $\BSO_q$ in the map (\ref{eq:3.5}). We have seen that the universal vector bundle of $S(\tau)$ is isomorphic to the canonical kernel bundle over $S(\tau)$. On the other hand, the space $E_{q+1}\setminus\mathbf{0}$ consists of pairs $(L, v)$ of a subspace $L\subset \R^{\infty}$ of dimension $q+1$ and a non-zero vector in $L$. 
The universal vector bundle over $E_{q+1}\setminus\mathbf{0}$ is the vector bundle whose fiber over $(L,v)$ consists of vectors in $L$ orthogonal to $v$. Consequently the universal vector bundle over $E_{q+1}\setminus \mathbf{0}$ is canonically isomorphic to the restriction of the universal vector bundle over $S(\tau)$ to $E_{q+1}\setminus \mathbf{0}$ and therefore the inclusion (\ref{eq:3.5}) is a homotopy equivalence.   

\end{proof}

\section{Retraction theorem}

We recall that the spectrum $\A_1$ has the property that for every manifold $N$ the set $[N, \A_1]$ classifies proper oriented fold maps $M\to N$ of dimension $q$ up to cobordism. For each $t\ge 0$, let $A_{1,t}$ denote the space $S_t(\tau)$ where $\tau$ is the set of singularity types of regular and fold map germs. Then the $(t+q)$-th space of the spectrum $\A_1$ is the Thom space of the vector bundle $\pi_t^*\xi_t$ over $A_{1,t}$. Let $X_t$ denote the subspace of $A_{1,t}$ of fold map germs of index $0$. Then the Thom spaces of vector bundles $\pi_t^*\xi_t|X_t$ form a subspectrum of $\A_1$ denoted by $\mathbf{X}$. We will denote $\lim A_{1,t}$ and $\lim X_t$ by $A_1$ and $X$ respectively.

Our main result is based on the following somewhat surprising theorem whose proof occupies the rest of the section.

\begin{theorem}\label{l:1} Suppose that $q$ is even, then there is a retraction $r$ of $\A_1$ onto its subspectrum $\X$. 
\end{theorem}

We will show that the composition of the map 
\[
\pi|A_1: A_1 \longrightarrow \{-q\}\times \BSO  
\]
with the involution $I$ on $\BSO$ is homotopic relative to $X\subset A_1$ to a map with image in $\BSO_{q+1}$. This will readily imply Theorem~\ref{l:1}.

\begin{proof}[Proof of Theorem~\ref{l:1}]


Let $Z_t$ denote the subset of $A_{1,t}$ that consists of  fold map germs. Thus $A_{1,t}=Z_t\cup A_{0,t}$. Over $Z_t$ there are canonical kernel and cokernel bundles denoted by $K_{1,t}$ and $C_{1,t}$. We note that the kernel bundle $K_{1,t}$ is isomorphic to the normal bundle of $Z_t$ in $A_{1,t}$ (e.g., see \cite{Bo}). We denote the canonical kernel and cokernel bundles over $Z=\lim{Z_t}$ by $K_1$ and $C_1$ respectively. 

\begin{lemma}\label{l:14}
The map $\pi|Z$ coincides with the composition
\[
    Z \longrightarrow \{q\}\times \BSO \longrightarrow \{-q\}\times \BSO
\]
of a map classifying $K_1\ominus C_1$ and the involution $I$. In other words the map $\pi|Z$ is a stable vector bundle isomorphic to $-(K_1\ominus C_1)$. 
\end{lemma}
\begin{proof} As in the proof of Proposition~\ref{l:2}, we deduce that the map $\pi_t|Z_t$ classifies the stable vector bundle $(\pi_t|Z_t)^*(\xi_t\ominus \varepsilon^{t+q})$. On the other hand, there is an exact sequence of vector bundles 
\[
   0\longrightarrow K_{1,t}\longrightarrow (\pi_t|Z_t)^*\varepsilon^{t+q}\longrightarrow (\pi_t|Z_t)^*\xi_t\longrightarrow C_{1,t}\longrightarrow 0
\] 
over $Z_t$. Consequently, 
\[
   -(K_{1,t}\ominus C_{1,t})\simeq (\pi_t|Z_t)^*(\xi_t\ominus \varepsilon^{t+q}). 
\]
This implies the statement of the lemma.
\end{proof}

The following fact is known; we give its proof for the reader's convenience. 

\begin{lemma}\label{lemma1} If $q$ is even, then the cokernel bundle $C_1$ is trivial. 
\end{lemma}
\begin{proof} Let $x$ be a point in $Z$. It is represented by a map germ $f$. The second derivatives of $f|\ker(f)$ define a non-degenerate bilinear map of vector spaces
\begin{equation}\label{eq:form}
   K_1|x \otimes K_1|x\longrightarrow C_1|x
\end{equation}
over $x$; see \cite{Bo}. Thus we obtain a map $K_1\otimes K_1\to C_1$ of vector bundles. Since $K_1$ is of odd dimension, the line bundle $C_1$ is forced to be trivial. Indeed, in this case $C_1|x$ has a well-defined orientation with respect to which the quadratic form of (\ref{eq:form}) is definite positive over a subspace of $K_1$ of dimension greater than $q/2$.
\end{proof}

By Lemma~\ref{lemma1}, for even $q$, the canonical cokernel bundle $C_1$ is trivial. Thus, by Proposition~\ref{l:2} and \ref{l:14}, the maps $\pi|A_0$ and $\pi|Z$ are stable vector bundles isomorphic to $-K_0$ and $-(K_1\ominus \varepsilon)$ respectively. 

\begin{lemma}\label{lemma2} There is a vector bundle $\chi$ over $A_{1}$ of dimension $q+1$ such that $\chi|A_{0}$ is isomorphic to $K_{0}\oplus\varepsilon$, while $(\chi|Z)$ is isomorphic to $K_{1}$.
\end{lemma}
\begin{proof}
We will piece the vector bundles $K_{0}\oplus \varepsilon$ over $A_{0}$ and $K_{1}$ over $Z$ together to form a vector bundle $\chi$ over $A_{1}$ of dimension $q+1$. To this end we construct a vector bundle $\chi$ over a neighborhood $U\subset A_{1}$ of $Z$ so that $\chi$ is isomorphic to $K_{1}$ over $Z$ and $K_{0}\oplus \varepsilon$ over $U\setminus Z $. In fact we show that for a projection $p$ of $U$ onto $Z$, the vector bundle $\chi:=p^*K_{1}$ over $U$ satisfies the mentioned conditions. Every choice of an isomorphism of $\chi$ and $K_{0}\oplus\varepsilon$ over $U\setminus Z$ determines an extension of $\chi$ over $A_{1}$ satisfying the requirement of Lemma~\ref{lemma2}. Below we use an obvious isomorphism of $\chi $ and $K_{0}\oplus\varepsilon$ over $U \setminus Z$.

Let $Z'$ be the component of $Z$ of fold map germs of some index $i$. By Lemma~\ref{l:0},  the space $Z'$ is weakly homotopy equivalent to the classifying space $BG$ of the relative symmetry group of fold singularity of index $i$. Furthermore, by Lemma~\ref{l:2.4} and the fact that the normal bundle of $Z'$ in $U$ is isomorphic to $K_{1,t}$, the normal bundle of $Z'$ in $U$ is isomorphic to the Borel construction 
\[
      EG\times_{G} \R^{q+1} \longrightarrow BG,
\]
where $EG$ is the total space of the universal principle $G$-bundle over $BG$ and the action of $G$ on $\R^{q+1}$ is given by its right representation. In particular, for an appropriate choice of $BG$, we obtain a bijective map $\varphi$ of $EG\times_G\R^{q+1}$ onto an open tubular neighborhood $U_{i}$ of $Z'$. 

Let $f: \R^{q+1}\to \R$ be the standard Morse function of index $i$. In particular $0\in \R^{q+1}$ is the only singular point of $f$. Let $\tilde{p}$ be the projection $\R^{q+1}\to 0$. Then there is an isomorphism $T\R^{q+1}=\tilde{p}^*\mathop{\mathrm{ker}}(df|0)$. Over $\R^{q+1}\setminus \{0\}$ there is an isomorphism of vector bundles
\[
    \mathop\mathrm{ker}(df)\oplus \mathop{\mathrm{im}}(df)\approx T\R^{q+1}=\tilde{p}^*\mathop{\mathrm{ker}}(df|0),
\]
which, by the Borel construction, translates into an isomorphism 
\[
    EG\times_{G}\mathop\mathrm{ker}(df)\ \oplus\  EG\times_{G}\mathop{\mathrm{im}}(df)\ \approx\ EG\times_{G}T\R^{q+1}=EG\times_{G}\tilde{p}^*\mathop{\mathrm{ker}}(df|0),
\]
of vector bundles over the complement in $EG\times_{G}\R^{q+1}$ to the zero section.
In view of the map $\varphi$, this yields 
\[
    K_{0}\oplus p^*C_{1} \approx \chi = p^*K_{1} \quad \mathrm{over}\quad U_{i}\setminus Z'.
\]
In the case where $q$ is even this implies the desired requirement as $C_{1}$ is trivial. 
\end{proof}

Let $\xi_t$ denote the restriction $\xi|A_{1,t}$.

\begin{lemma}\label{l:3.5} The map $\pi|A_{1}$ is a stable vector bundle of dimension $-q$ isomorphic to $-\chi\oplus\varepsilon$.
\end{lemma}
\begin{proof} It suffices to show that for $t\ge 1$ we have $\pi_t^*(\xi_t\ominus \varepsilon^{t+q})\simeq -\chi_t\oplus\varepsilon$, or equivalently, that the stable vector bundle $\chi_t\oplus \xi_t$ is trivial over $A_{1,t}$. Let $J_t$ denote the canonical image bundle over $Z_t$. Then there is a canonical trivialization of the vector bundle $\chi_t\oplus J_t$ which over a map germ $f\in Z_t$ is given by the isomorphism 
\[
       T_0\R^{q+t}\stackrel{\approx}\longrightarrow \mathop{\mathrm{ker}}(df|T_0\R^{q+t})\oplus \mathop{\mathrm{coim}}(df|T_0\R^{q+t})\stackrel{\approx}\longrightarrow \chi_t\oplus J_t|f.  
\]  
Let $p_t$ denote the projection of the open neighborhood $U_t$ of $Z_t$ in $A_{1,t}$ onto $Z_t$. It pulls back the above trivialization to a trivialization of $p_t^*(\chi_t\oplus J_t)$ over $\partial U_t$. In other words over each point $f$ in $\partial U_t$ we obtain a trivialization 
\[
    T_0\R^{q+t}\stackrel{\approx}\longrightarrow \chi_t\oplus J_t|f \stackrel{\approx}\longrightarrow [K_{0,t}\oplus p_t^*C_{1,t}]\oplus p^*J_t|f \stackrel{\approx}\longrightarrow K_{0,t}\oplus \xi_t|f,
\]
since $J_t\oplus C_{1,t}\approx \xi_t$. 
This trivialization coincides with the canonical trivialization of $K_{0,1}\oplus \xi_t$ over $A_{0,t}$, and therefore, it extends to a trivialization of $\chi_t\oplus \xi_t$ over $A_{1,t}$. 
\end{proof}

Since the relative symmetry group of the fold map germ of index $0$ is isomorphic to the group $\SO_{q+1}$, by Lemma~\ref{l:0} the cohomology theory $\mathbf{X}$ corresponds to a stable vector bundle over $X\simeq \BSO_{q+1}$. 

\begin{lemma}\label{l:4.6} The map $\pi|X$ is given by the composition 
\[
    {\BSO}_{q+1}\longrightarrow \{q\}\times \BSO\longrightarrow \{-q\}\times\BSO
\]
of the inclusion classified by $\chi\ominus \varepsilon$ and the involution $I$. 
\end{lemma}
\begin{proof} The restriction of the vector bundle $\chi$ to $X$ is the universal vector bundle over $\BSO_{q+1}$. On the other hand $\chi|X=K_1$. Consequently the statement of the lemma follows from Lemma~\ref{l:14}. 
\end{proof}

Now we are in position to complete the proof of Theorem~\ref{l:1}. By Lemma~\ref{l:3.5}, the composition 
\begin{equation}\label{eqn:1}
     I\circ \pi|A_1\colon A_1\longrightarrow \{-q\}\times \BSO\longrightarrow \{q\}\times \BSO
\end{equation}
is a stable vector bundle of dimension $q$ isomorphic to $\chi\ominus\varepsilon$. Since $\chi$ is a (non-stable) vector bundle of dimension $q+1$, the composition (\ref{eqn:1}) factors through a composition
\[
   A_1\stackrel{r}\longrightarrow {\BSO}_{q+1}\longrightarrow \{q\}\times \BSO
\]
where the latter map is the canonical inclusion classifying $\xi_{q+1}\ominus\varepsilon$. We note that by Lemma~\ref{l:4.6} the restriction of the former map $r$ to $\lim X_t$ is the identity map. Thus, the map $r$ is a desired retraction. 

This completes the proof of Theorem~\ref{l:1}.
\end{proof}

\section{The splitting theorem}

For $i=0,...,\left\lfloor \frac{q+1}{2}\right\rfloor$, let $\B_i$ denote the spectrum $\Sigma_{\tau}/\A_0$ where $\tau$ stands for the set of the fold singularity of index $i$ and the singularity of the regular map germ. Then the spectrum $\A_1/\mathbf{A}_0$ is a wedge sum of spectra $\vee\mathbf{B}_i$ indexed by $i=0,\dots, \left\lfloor \frac{q+1}{2}\right\rfloor$. 

Let us recall that the spectrum $\Sigma_{\tau}$, where $\tau$ stands for the set of the fold singularity of index $0$ and the singularity of the regular map germ, is equivalent to the spectrum $\mathbf{X}$ (see Lemma~\ref{l:3.6}). Hence, Theorem~\ref{l:1} implies that in the case where $q$ is even the cofiber sequence of spectra
\begin{equation}\label{l:3}
    \X \stackrel{j}\longrightarrow \A_1 \longrightarrow \vee_{i>0}\B_i
\end{equation} 
splits, i.e., the following theorem takes place.

\begin{theorem}\label{th:1}
If the dimension $q$ of maps is even, then the spectrum $\A_1$ is equivalent to the spectrum $\vee_{i>0}\B_i\vee \X$.
\end{theorem}

In the case where $q$ is odd, the proof of Theorem~\ref{l:1} only implies that 
there is a cofiber sequence of spectra
\begin{equation}\label{e:2}
    \vee_{i\ne 0,s}\B_i\vee \X \longrightarrow \A_1 \longrightarrow \B_s
\end{equation}
where $s=\frac{q+1}{2}$. Indeed, let $\A_1'$ denote the spectrum $\Sigma_{\tau}$, where $\tau$ stands for the set of fold singularities of index $\ne s$ and the singularity type of a regular germ. Then, by the proof of Theorem~\ref{l:1}, there is a retraction $\mathbf{A}'_1\to \mathbf{X}$, which implies that the spectrum $\A_1'$ is equivalent to $\vee_{i\ne 0,s}\B_i\vee \X$. Now the cofiber sequence (\ref{e:2}) follows from the observation that $\A_1/\A_1'$ is equivalent to $\B_s$. 

Let $\mathbf{Y}$ denote the subspectrum of $\A_1$ defined as $\Sigma_{\tau}$, where $\tau$ consists of fold singularities of indexes $0$ and $s$ and the singularity type of a regular map germ. We note that there is a projection of $\A_1$ onto $\vee_{i\ne 0,s}\B_i$. In view of the first map in the sequence (\ref{e:2}) we deduce the following theorem. 

\begin{theorem}\label{th:4} In the case where the dimension $q$ of maps is odd, the spectrum $\A_1$ is equivalent to $\vee_{i\ne 0,s}\B_i\vee \mathbf{Y}$.
\end{theorem} 

Theorems~\ref{th:1} and \ref{th:4} are equivalent to Theorem~\ref{main} in the Introduction. In the rest of the section we will prove a cohomology splitting of the spectrum $\mathbf{A}_1$.  

The spectrum $\mathbf{Y}$ in Theorem~\ref{th:4} can be described in terms of cofiber sequences
\[
      \X \longrightarrow \mathbf{Y} \longrightarrow \B_s
\]
and
\[
     \A_0 \longrightarrow \mathbf{Y} \longrightarrow \B_0\vee \B_s.
\]

It is not known if the spectrum $\mathbf{Y}$ splits off the term $\mathbf{B}_s$. On the other hand, the cohomology groups of the spectra $\A_1$ split both in the cases of even and odd dimension as follows. 
 
\begin{theorem}\label{th:2} For any coefficient field, the cohomology group $H^*(\A_1)$ is isomorphic to the sum $\oplus_{i>0}H^*(\B_i)\oplus H^*(\X)$. 
\end{theorem}
\begin{proof} The equation~(\ref{l:3}) gives rise to a long exact sequence 
\[
    \cdots\longrightarrow\oplus_{i>0} H^*(\B_i)\longrightarrow H^*(\A_1)\stackrel{j^*}\longrightarrow H^*(\X)\longrightarrow\cdots.
\]
To prove Theorem~\ref{th:2} it suffices to show that the homomorphism $j^*$ induced by the inclusion of $\X$ into $\A_1$ admits a right inverse. In the case where the dimension $q$ is even the existence of a right inverse of $j^*$ follows from Theorem~\ref{th:1}.

Suppose now that $q$ is odd. 

Let us recall that $\X$ is the Thom spectrum of a stable vector bundle over $\BSO(q+1)$. The cohomology ring of $\BSO(q+1)$ for $q$ odd with coefficients in a field $\mathfrak{K}$ of characteristic $\ne 2$ is isomorphic to 
\[
   H^*(\BSO(q+1); \mathfrak{K})\simeq \mathfrak{K}[p_1, p_2, ..., p_{\frac{q+1}{2}}, e]/(e^2-p_{\frac{q+1}{2}}),
\]
where $p_i$ and $e$ are the $i$-th Pontrjagin and Euler classes respectively. Consequently, as an additive group it is isomorphic to the sum of two polynomial groups 
\[
   H^*(\BSO(q+1); \mathfrak{K})\simeq \mathfrak{K}[p_1, p_2, ..., p_{\frac{q+1}{2}}]\oplus e\smallsmile \mathfrak{K}[p_1, p_2, ..., p_{\frac{q+1}{2}}].
\]

By the Thom isomorphism the spectrum cohomology group of $\X$ is isomorphic to a sum of groups of polynomials $P\oplus e\smallsmile P$, where $P$ is a polynomial group on Pontrjagin classes $p_1, ..., p_{\frac{q+1}{2}}$ if the coefficients are chosen to be in a field of characteristic $\ne 2$, and the polynomial group on Stiefel-Whitney classes $w_2,..., w_{q+1}$ otherwise. 

We will only give an argument in the case where the coefficient field is chosen to be of characteristic $0$. The argument in other cases is similar. 

We define $(j^*)^{-1}|P$ to be the composition
\[
     P\longrightarrow H^*(\BSO) \longrightarrow H^*(\MSO) \longrightarrow H^*(\A_1)
\] 
of the canonical inclusion, the Thom isomorphism, and the homomorphism induced by the canonical map $\A_1\to \MSO$ of Thom spectra; and the homomorphism $(j^*)^{-1}|e\smallsmile P$ to be the composition
\[
   e\smallsmile P \longrightarrow P\longrightarrow H^*({\BSO}_{q+1})\longrightarrow H^*(\B_0) \longrightarrow H^*(\A_1)
\]
of the canonical isomorphism, the canonical inclusion, the Thom isomorphism, and the homomorphism induced by the projection $\A_1\to \B_0$.  

It remains to show that for each $x\in P\oplus e\smallsmile P$, the restriction of the class $(j^*)^{-1}(x)$ to $\mathbf{X}$ is $x$. In the case where $x\in P$, this immediately follows from the definition of $(j^*)^{-1}(x)$. In the case where $x\in e\smallsmile P$, this follows from Lemma~\ref{l:2.4}.  
\end{proof}

\begin{remark} In \cite{Sa1} we study invariants of singular maps. In notation adopted in \cite{Sa1}, the classes in the image of $(j^*)^{-1}|P$ are called the generalized Miller-Morita-Mumford classes, while the classes in the image of $(j^*)^{-1}|P$ are classes $I_{Q,0}$. 
\end{remark}

\section{Geometric meaning of cohomology theories $\B_i$}\label{s:gm}

For each integer $i$ in the interval $[0, \left\lfloor \frac{q+1}{2}\right\rfloor]$, the cobordism class of a proper fold map $f: M\to N$ of dimension $q$ gives rise to a homotopy class $\mathfrak{o}_i(f)$ in $\mathbf{B}_i^0(N)$. Indeed, there is a projection $\A_1\to \B_i$ and $\mathfrak{o}_i(f)$ is defined to be the value under the induced map 
\[
    [\mathbf{S}\wedge N_+, \A_1] \longrightarrow [\mathbf{S}\wedge N_+, \B_i]
\] 
of the homotopy class $\mathop\mathrm{PT}(f)$ classifying by the Pontrjagin-Thom construction \cite{Sa} the cobordism class of the fold map $f$. 
We note that the class $\mathfrak{o}_i(f)$ is a {\it partial} obstruction to the existence of a homotopy of $f$ to a fold map without fold singularities of index $i$. In other words, if $f$ is homotopic to a fold map without fold singularities of index $i$, then, by the Pontrjagin-Thom construction, the class $\mathfrak{o}_i(f)$ is trivial. On the other hand, a priori, even if $\mathfrak{o}_i(f)$ is trivial, the fold singularity of index $i$ still may be essential for the homotopy class of $f$.  

\begin{remark} In slightly different terms invariants of fold maps given by $\mathfrak{o}_i(f)$ are defined and studied by Kalmar in \cite{Ka5} and his other papers. 
In a more general setting these obstructions are defined in \cite{Sa} and \cite{Sz}. Similar obstructions have been also defined earlier by Koschorke in \cite{Ko}
for morphisms of vector bundles. An equivariant
cohomology version of obstructions $\mathfrak{o}_i(f)$ is defined in \cite{FR}.
\end{remark}

To begin with let us explain the relation of $\B_i$ to the classifying space of an appropriate Lie group as has been promised in the Introduction. Let $BG_i$ denote the classifying space of the relative symmetry group $G_i$ of the fold singularity of index $i$ (see the list of groups $G_i$ for each $i$ in section~\ref{s:2}). 

\begin{proposition}\label{p:4.3} The spectrum $\B_i$ is equivalent to the suspension spectrum of the Thom space of $C_1|BG_i$. In particular, if either $q$ is even, or $q$ is odd and $i\ne \frac{q+1}{2}$, then the spectrum $\B_i$ is equivalent to the suspension spectrum of the topological space $S^1\wedge \BSO(i, q+1-i)_+$. 
\end{proposition}
\begin{proof} The spectrum $\B_i$ is the spectrum of a stable vector bundle over $BG_i$; see Lemma~\ref{l:0}. This stable vector bundle is the sum of the universal vector bundle over $BG_i$ of dimension $q+1$ (see Lemma~\ref{l:2.4}) and the stable vector bundle $\pi|BG_i$ which by Lemma~\ref{l:14} is isomorphic to $-(K_1\ominus C_1)|BG_i$. Thus the stable vector bundle of the Thom spectrum $\mathbf{B}_i$ is the one of the line bundle $C_1|BG_i$. In the case where $q$ is even, or $q$ is odd and $i\ne \frac{q+1}{2}$ the line bundle $C_1|BG_i$ is trivial (see Lemma~\ref{lemma1}) and $G_i=\SO(i, q+1-i)$.  
\end{proof}


\begin{definition} A {\it framed immersion} (or a {\it framed $G_i$-immersion}) into a closed manifold $N$ is defined to be a dimension $-1$ immersion $h: S_i\to N$ together with a continuous map $j: S_i\to BG_i$ and an isomorphism $\gamma$ of $j^*C_1$ and the normal bundle of $h$. 
The notion of {\it cobordism classes} of framed immersions can be defined in a usual way. 
\end{definition}

Now we are in a position to give a geometric interpretation of classes $\mathfrak{o}$ in the group $\mathbf{B}_i^0(N)$.

\begin{proposition}\label{l:15}
For a smooth orientable manifold $N$, classes in $\mathbf{B}_i^0(N)$ are in bijective correspondence with cobordism classes of framed immersions into $N$. 
\end{proposition}
\begin{proof}
The statement of Proposition~\ref{l:15} follows from the Pontrjagin-Thom construction and the Smale-Hirsch h-principle for immersions. Indeed, by 
the Pont\-rjagin-Thom construction, a class $\mathfrak{o}$ in $\mathbf{B}_i^0(N)$ gives rise to a quadruple $(S_i, h, j, \varphi)$ of 
\begin{itemize}
\item a smooth closed manifold $S_i$ of dimension $\dim N-1$,
\item a smooth map $h: S_i\to N$,
\item a continuous map $j: S_i\to BG_i$, and 
\item a decomposition of the normal bundle $\nu_h=h^*TN\ominus TS_i$ into a Whitney sum of two stable vector bundles by means of an isomorphism
\[
\varphi\colon \nu_h \longrightarrow j^*(K_1)\oplus j^*(- K_1\oplus C_1)
\] 
of $\nu_h$ and the pullbacks of two stable vector bundles over $BG_i$. 
\end{itemize}

In particular, 
there is an isomorphism of vector bundles
\[
    TS_i\oplus j^*C_1 \approx h^*TN
\] 
over $S_i$; this isomorphism is an isomorphism of unstable bundles because of a dimensional argument. Hence, by the Smale-Hirsch Theorem, there is an immersion $S_i\to N$ with normal bundle isomorphic to $j^*C_1$. Consequently, the class $\mathfrak{o}$ gives rise to a framed immersion into $N$. Thus we get a map assigning to each class $\mathfrak{o}$ a framed immersion. It is straightforward to show that this map induces a bijective correspondence between homotopy classes in $\mathbf{B}_i^0(N)$ and cobordism classes of framed immersions into $N$. 
\end{proof}

\section{Examples}

In this section we consider the cobordism group of fold maps into a sphere $S^n$ of dimension $n>1$. It is slightly more convenient to consider the reduced group $\tilde{\mathbf{A}}_1^0(S^n)$ rather than the unreduced one. We recall that the reduced group $\tilde{\A}_1^0(S^n)$ is defined to be 
\[
   \tilde{\A}_1^0(S^n)\colon = [\mathbf{S}\wedge S^n, \A_1]=\pi_n(\A_1).
\] 
The obstructions $\mathfrak{o}_i(f)$ are now replaced with reduced obstructions $\tilde{\mathfrak{o}}_i(f)\in \tilde{\B}_i^0(S^n)$ which are images of $[f]$ under the coefficient homomorphisms $\tilde{\A}_1^0(S^n)\to \tilde{\B}_i^0(S^n)$.

By Theorem~\ref{th:1} in the case of cobordism groups of equidimensional fold maps (i.e., in the case $q=0$), the spectrum $\mathbf{A}_1$ is equivalent to the spectrum $\mathbf{X}$, which in its turn is equivalent to the sphere spectrum $\mathbf{S}$. We immediately deduce the following proposition. 

\begin{proposition} The reduced cobordism group of equidimensional fold maps into a sphere $S^n$ is isomorphic to the stable homotopy group $\pi^{st}_{n}$ provided that $n>1$. 
\end{proposition}

By Theorem~\ref{th:1} in the case of cobordism groups of fold maps of dimension $q=2d$, the spectrum $\mathbf{A}_1$ is equivalent to $\mathbf{X}\bigvee \vee_i \mathbf{B}_i$ where the index $i$ ranges from $1$ to $d$. This allows us to express the reduced cobordism groups of fold maps of spheres in terms of stable homotopy groups of classifying spaces of Lie groups. 

\begin{proposition}\label{p:5.2} The reduced cobordism group of fold maps of dimension $q=2d$ into a sphere $S^n$, with $n>1$, admits a splitting isomorphism of groups
\[    
\sigma\colon   \tilde{\mathbf{A}}_1^0(S^n)\simeq \tilde{\mathbf{X}}^0(S^n)\oplus [\mathop{\oplus}_{i=1}^d \tilde{\mathbf{B}}_i^0(S^n)]\simeq  \tilde{\mathbf{X}}^0(S^n)\oplus [\mathop{\oplus}_{i=1}^d \pi^{st}_{n-1}\oplus \pi^{st}_{n-1}(BG_i)]
\]
where $BG_i$ is the relative symmetry group of the fold singularity of index $i$. 
\end{proposition}
\begin{proof}
By Proposition~\ref{p:4.3} the spectrum $\mathbf{B}_i$ is equivalent to the suspension spectrum of the space $S^1\wedge (BG_i)_+$, where $(BG_i)_+$ is $BG_i\sqcup \{pt\}$. On the other hand 
there is an obvious inclusion $S^1\to S^1\wedge (BG_i)_+$ and a projection $S^1\wedge (BG_i)_+\to S^1$ whose composition is the identity map of $S^1$. Hence, for every manifold $N$, the group $\mathbf{B}_i^0(N)$ splits off $\mathbf{S}^1(N)$, where $\mathbf{S}$ is the sphere spectrum. More precisely, 
\[
   \tilde{\B}_i^0(S^n)\simeq \mathop{\oplus}_{i=1}^d [\pi^{st}_{n-1}\oplus \pi^{st}_{n-1}(BG_i)]
\]
\end{proof}

\begin{remark} Our proof of Proposition~\ref{p:5.2} shows that the cobordism group of fold maps of manifolds of dimension $n+q$ into an arbitrary manifold $N$ of dimension $n$ splits off the direct sum of $\left\lfloor \frac{q-1}{2}\right\rfloor$ copies of the $1$-st cohomology group $\mathbf{S}^1(N)$. This recoves (and improves) the main theorem in \cite{Ka2}.  
\end{remark}

\begin{remark} There is a fairly simple geometric argument \cite{Ka6} showing that the map $\sigma\times \mathfrak{o}_0$ is an injective homomorphism of $\tilde{\mathbf{A}}_1^0(S^n)$. Here $\mathfrak{o}_0$ stands for the homomorphism $\tilde{\mathbf{A}}_1^0(S^n)\to \tilde{\mathbf{B}}_0^0(S^n)$ that takes the class represented by a map $f$ onto $\mathfrak{o}_0(f)$.  
In the case where $-q$ is odd, a similar statement takes place for cobordism groups of so-called framed fold maps \cite{Ka6}.    
\end{remark}

In the case of cobordism groups of fold maps of dimension $q=2d+1$, by Theorem~\ref{th:4}, there is an equivalence of spectra $\mathbf{A}_1$ and $\mathbf{Y} \bigvee\vee_i \mathbf{B}_i$ with $i$ ranging from $1$ to $d$. As above we deduce the following proposition. 

\begin{proposition}\label{c:5.3} For the reduced cobordism group of fold maps of dimension $q=2d+1$ into $S^n$, with $n>1$, there is a splitting isomorphism of groups
\[    
   \tilde{\mathbf{A}}_1^0(S^n)\simeq \tilde{\mathbf{Y}}^0(S^n)\oplus [\mathop{\oplus}_{i=1}^d \tilde{\mathbf{B}}_i^0(S^n)]\simeq  \tilde{\mathbf{Y}}^0(S^n)\oplus [\mathop{\oplus}_{i=1}^d \pi^{st}_{n-1}\oplus \pi^{st}_{n-1}(BG_i)]
\]
where $BG_i$ is the relative symmetry group of the fold singularity of index $i$. 
\end{proposition}

The term $\tilde{\mathbf{Y}}^0(S^n)$ in Proposition~\ref{c:5.3} fits the exact sequence 
\[
   \tilde{\mathbf{A}}_0^0(S^n)\longrightarrow \tilde{\mathbf{Y}}^0(S^n) \longrightarrow \tilde{\mathbf{B}}^0_0(S^n)\oplus \tilde{\mathbf{B}}^0_s(S^n), \quad \mathrm{with}\quad s=\frac{q+1}{2},
\]
associated to the cofiber sequence 
\[
    \mathbf{A}_0\longrightarrow \mathbf{Y}\longrightarrow \mathbf{B}_0\vee \mathbf{B}_s.
\]
In some cases this allows us to simplify computations of $\mathbf{Y}(S^n)$. 

\begin{corollary} The reduced cobordism class of a fold map $f$ of dimension $1$ into $S^4$ is determined by reduced obstructions $\tilde{\mathfrak{o}}_0(f)$ and $\tilde{\mathfrak{o}}_1(f)$. 
\end{corollary}
\begin{proof}
In the case of maps into $S^4$ we obtain an exact sequence
\[
   \pi^{st}_5 \longrightarrow \tilde{\mathbf{Y}}({S}^4) \longrightarrow \tilde{\mathbf{B}}_0^0({S}^4)\oplus \tilde{\mathbf{B}}_1^0({S}^4)\longrightarrow \pi^{st}_4.
\] 
It is known that $\pi^{st}_4=\pi^{st}_5=0$. Thus $\tilde{\mathbf{Y}}^0(S^4)$ is isomorphic to the sum of two reduced cobordism groups of framed immersions.   
\end{proof}

\section{Splitting of the spectrum $\mathbf{hW}$ of Madsen and Weiss}\label{s:MW}

In this section we will assume that $q$ is even and clarify the relation of our result to the homotopy cofiber sequence
 of spectra
\begin{equation}\label{eq:t}
  \mathbf{hV}\longrightarrow \mathbf{hW}\longrightarrow \mathbf{hW_{loc}}
\end{equation}
considered by Madsen and Weiss (for the notation we refer the reader to the paper \cite{MW}; the spectra equivalent to those in the cofibration (\ref{eq:t}) are also defined below). Namely, by \cite[Lemma 3.1.1]{MW} there is a splitting of the spectrum $\mathbf{hW_{loc}}$, 
\[
    \mathbf{hW_{loc}}\simeq \B_0\vee \B_1\vee \cdots \vee \B_{q+1}.
\] 
We will show that the cofibration (\ref{eq:t}) is almost trivial. More precisely, the spectrum $\mathbf{hW}$ splits as
\[
    \mathbf{hW}\simeq \B_0\vee \B_1\vee \cdots \vee \B_q\vee \mathbf{X},
\]
and the second map in the cofibration (\ref{eq:t}) takes each of the first $q$ factors of $\mathbf{hW}$ identically onto the corresponding factor of $\mathbf{hW_{loc}}$ while $\mathbf{X}$ and $\B_{q+1}$ fit the homotopy cofibration sequence of spectra
\[
    \mathbf{hV}\longrightarrow \mathbf{X}\longrightarrow \B_{q+1}. 
\]

We recall that the spectra $\A_1$ and $\A_0$ coincide with the spectra of the stable vector bundles $\pi|A_1$ and $\pi|A_0$ respectively. Let $A_1'$ denote the non-Hausdorff space obtained from $A_1\sqcup A_1$ by identifying the corresponding points in the two copies of $A_0$. Then there is a projection 
\[
   \rho\colon A_1'\longrightarrow A_1
\]
such that its restrictions to $\rho^{-1}(A_0)$ and $\rho^{-1}(A_1\setminus A_0)$ are a bijection and a trivial double cover respectively. For each stratum $\Sigma$ in $A_1$ of Morse map germs of index $i$, the inverse image $f^{-1}(\Sigma)$ has two components in $A_1\sqcup A_1/\sim$. These components are said to be the strata in $A'_1$ of Morse map germs of indexes $i$ and $q+1-i$ respectively. 

The composition $\pi\circ \rho|A_1'$ itself is a stable vector bundle over $A_1'$ whose spectrum is denoted by $\A_1'$. 

\begin{lemma}\label{l:spl} The spectra $\A_1'$ and $\mathbf{hW}$ are equivalent. 
\end{lemma}
\begin{proof} In notation of the paper \cite{MW}, the spectrum $\mathbf{hW}$ is the spectrum of a stable vector bundle over a space $\lim \GW(q+1,n)$. The spaces $\lim \GW(q+1,n)$ and $A_1'$ are stratified. Each of the two spaces has a stratum homotopy equivalent to $\BSO_q$, which corresponds to regular map germs; and one stratum homotopy equivalent to $\BSO(i, q+1-i)$, which corresponds to the singularity type of Morse map germs on $(\R^{q+1},0)$ of index $i=0,...,q+1$ (see Lemma~\ref{l:0} and \cite[Lemma 3.1.1]{MW}). 

The normal bundle of each stratum of Morse map germs is the universal vector bundle (see Lemma~\ref{l:2.4} and \cite[page 20]{MW}), and therefore there is a map from a neighborhood of strata of Morse map germs in $A_1'$ to a corresponding neighborhood in $\lim \GW(q+1,n)$. This map can be extended over the remaining stratum homotopy equivalent to $\BSO_q$ by means of the map classifying the universal vector bundle over $\BSO_q$. 
\end{proof}

In fact the proof of Lemma~\ref{l:spl} shows that there is a commutative diagram of cofibrations

\[
\begin{CD}
\mathbf{hV} @>>> \mathbf{hW} @>>> \mathbf{hW_{loc}} \\
@V\simeq VV  @V\simeq VV @V\simeq VV\\
\A_0 @>>> \A_1' @>>> \A_1'/\A_0. 
\end{CD}
\]

Let $\mathbf{X'}$ be a subspectrum of the spectrum $\A_1'$ that corresponds to regular map germs and Morse map germs of index $q+1$. In other words $\X'$ is a spectrum of a stable vector bundle over a subspace of $A_1'$ of regular map germs and Morse map germs of index $q+1$. We observe that $\mathbf{X'}$ is equivalent to the spectrum $\mathbf{X}$. There is a homotopy cofiber sequence of spectra
\begin{equation}\label{eq:8.1}
     \mathbf{X} \simeq \mathbf{X'}\longrightarrow \A_1'\longrightarrow \B_0\vee \cdots \vee \B_q. 
\end{equation}
By an argument in the proof of Theorem~\ref{th:1}, there exists a map $\A_1'\to \mathbf{X'}$ that is right inverse to the map $\mathbf{X'}\to \A_1'$ in the homotopy cofiber sequence (\ref{eq:8.1}). Thus, we conclude that 
\[
   \mathbf{hW}\simeq \A_1'\simeq \B_0\vee \cdots \vee \B_q\vee \mathbf{X}. 
\]
\begin{remark}\label{r:8.1} The comparison of the homotopy fiber of the maps $\A'_1\to \A'_1/\A_0$ and $\A_1\to \A_1/\A_0$ shows that the homotopy fiber of the map $\mathbf{hW}\to \mathbf{hW_{loc}}$ is equivalent to the homotopy fiber of the map $\A_1\to \A_1/\A_0$. 
\end{remark}

\section{Further applications}\label{s:10}

\subsection{Cobordism groups of Morin maps}
As has been mentioned all invariants of cobordism groups of fold maps are either cobordism classes of framed immersions or cohomology classes in $\mathbf{X}(\cdot)$ or $\mathbf{Y}(\cdot)$ related to diffeomorphism groups of manifolds. In general these invariants are too difficult to compute. On the other hand, in a forthcoming paper \cite{Sa1} we use the results of the current paper to study rational invariants of cobordism groups of so-called Morin maps, which are maps with singularities of a more general type than that of fold ones.

\end{document}